\documentclass{amsart}
\usepackage{hyperref}
\usepackage{graphicx}
\usepackage{amscd}
\usepackage{verbatim}
\usepackage[dvipsnames]{xcolor}
\usepackage{tikz}
\usetikzlibrary{arrows.meta}
\usetikzlibrary{bending}
\usepackage{bm}
\usepackage{asymptote}
\newtheorem{theorem}{Theorem}[section]
\newtheorem{lemma}[theorem]{Lemma}

\theoremstyle{definition}

\newtheorem{proposition}[theorem]{Proposition}

\newtheorem{corollary}[theorem]{Corollary}

\theoremstyle{remark}
\newtheorem{remark}[theorem]{Remark}

\numberwithin{equation}{section}

\begin{document}

\title{On the monodromy of elliptic surfaces}

\author{Genival Da Silva Jr.}

\address{\newline Department of Mathematics \newline Imperial College London\newline London, SW7 2AZ, UK}
\email{jrriemann@gmail.com}


\begin{abstract}
There have been several constructions of family of varieties with exceptional monodromy group \cite{DR},\cite{Y}. In most cases, these constructions give Hodge structures with high weight(Hodge numbers spread out). N. Katz was the first to obtain Hodge structures with low weight(Hodge numbers equal to $(2,3,2)$) and geometric monodromy group $G_2$. In this article I will give an explicit proof of Katz's result by finding all the monodromies in a given basis.
\end{abstract}

\maketitle

\section{Introduction}
In \cite{KATZ}, Nicholas Katz studies the appearance of $G_2$ as the monodromy group of a family of elliptic surfaces. Starting with an elliptic curve $E/k$ and a `seven point sheaf' on $E$ (geometrically irreducible lisse sheaf $\mathcal{F}$ of rank 2 on a dense open set $j:U\subset E$), he wonders what are the groups $G_{geom,N},G_{arith,N}$, where $N=j_{*}\mathcal{F}(1/2)[1]$.

One way to obtain such $\mathcal{F}$ is to consider $E$ as a double covering $x:E\rightarrow \mathbb{P}^1$ of $\mathbb{P}^1$. Then our $\mathcal{F}$ is obtained by pulling back a `four point sheaf' $\mathcal{G}$ on $\mathbb{P}^1$, and we can get such $\mathcal{G}$ by taking $R^1\pi_{*}\overline{\mathbb{Q}_l}(1/2)$ for an elliptic surface $\pi: \mathcal{E}\rightarrow \mathbb{P}^1$.

Over $\mathbb{C}$, up to isogeny, there are only 4 elliptic surfaces $\pi: \mathcal{E}\rightarrow \mathbb{P}^1$, they are:
\begin{equation}\label{families}
\begin{split}
    & y^2 = -x(x-1)(x-\lambda^2),\; \lambda \neq 0,\pm 1, \infty\\
    & y^2 = 4x^3 + ((\lambda + 2)x + \lambda)^2, \; \lambda \neq 0,1,-8,\infty\\
    & y^2 = 4x^3 + (\lambda^2+6\lambda -11)x^2 + (10 - 10\lambda)x + 4\lambda -3, \; \lambda \neq 0,\infty \text{ and } \lambda^2+11\lambda-1 \neq 0\\
    & y^2 = 4x^3 + (3\lambda x + 1)^2,  \; \lambda \neq \infty,\lambda^3 \neq 1
\end{split}
\end{equation}

For each one these 4 families, we can associate a monic cubic polynomial $f$, whose roots are the `bad' values of $\lambda$ described above. Moreover:

\begin{theorem}\cite[theorem 4.1]{KATZ}
For each one of the four families above, there is an explicit nonzero integer polynomial $P[T]\in\mathbb{Z}[T]$ with the following property. For each finite field $k$ in which $l$ is invertible, and for each $t\in k$ at which $P(t)\neq 0$ in $k$, the equation $$E_t : y^2=tf(x) + t^2$$ defines an elliptic curve over $k$, and the $N$ gotten by pulling back $\mathcal{G}(x)$ has $$G_{geom,N}=G_{arith,N}=G_2$$.
\end{theorem}

The proof of the theorem above involves the analysis of the sheaf $$\mathcal{H}:= R^1\rho_{*}(\mathcal{G}(1/2)\otimes \mathcal{L}_{\chi_2(tf(x) + t^2)})$$ where $\rho:\mathbb{P}^1\rightarrow S$, $S$ a punctured affine $t$ line, $\mathcal{L}_{\chi_2}$ is the Kummer sheaf attached to the character $\chi_2$ of $k^{\times}$. Katz\cite[theorem 5.1]{KATZ} then proves that for the first three (but not the fourth) families above the sheaf $\mathcal{H}$ has $G_{geom}=G_{arith}=G_2$. In the rest of this paper I will give an alternative proof of the latter fact by explicitly finding the monodromies matrices.

\begin{theorem}
The geometric monodromy group of the first three families considered above is $G_2$. 
\end{theorem}

\subsection*{Acknowledgements} I thank my PhD advisor Matt Kerr for sharing his ideas with me, and also acknowledge the travel support from NSF FRG Grant 1361147, the support form CNPq program \textit{Science without borders} and the support from ERC Grant 682603.

\section{The explicit construction of the family}
For the sake of simplicity, I will work in this section with the first of the 3 families with monodromy group $G_2$, but the exact same approach applies to the remaining ones. We start off, by constructing the family of surfaces mentioned above.

Let $\mathcal{E}\rightarrow\mathbb{P}^1:y^2=x(x-1)(x-z^2)$ be a rational elliptic surface with singular fibers at $z=-1,0,1,\infty$. For $t\neq 0,\pm \frac{2}{3\sqrt{3}} , \infty$, take a base change by:
\begin{equation}
    E_t\rightarrow \mathbb{P}^1:w^2=tz(z-1)(z+1)+t^2
\end{equation}
The result is a family of elliptic surfaces $X_t\rightarrow E_t$ with 7 singular fibers on each surface, as described below:
\begin{equation}
    \begin{array}{ccc}
        \mathcal{X} & \hookleftarrow & X_t\\
        \downarrow \pi &  & \downarrow \pi_t\\
        \mathcal{E} & \hookleftarrow & E_t\\
        \downarrow &  & \downarrow\\
        \mathbb{P}^1 & \hookleftarrow & \{ t \}
    \end{array}
\end{equation}

\begin{proposition}\label{trd}
For each $X_t$ we have $dim(H^2_{tr}(X_t)) \leq 7$.
\end{proposition}
\begin{proof} 
Set $X := X_t , E := E_t, \pi := X \to E$. We have that $X$ has 4 singular fibers of type $I_2$, 2 of type $I_4$ and one of type $I_8$. We use known formulas for the hodge numbers of Elliptic Surfaces, see \cite{SS}. We have: $$b_2(X) = 2p_g(X) + 10\chi(X) + 2g(E)$$ where $p_g(X)=\dim H^{2,0}(X)$ is the geometric genus of $X$, $\chi(X)$ is the Euler characteristic and $g(E)$ is the genus of $E$. In our particular case we have:
$$g(E)=1,\;\chi(X)=2;\;p_2(X)=\chi(X)-1+g(E)=2$$
which amounts to $b_2(X)=26$. Since the components of the singular fibers give algebraic cycles, and we have $2+4+6+7=19$ of them, the dimension of transcendental cycles can be at most 7, and it is 7 if we can find 7 linearly independent transcendental  cycles.
\end{proof}
\begin{remark}
In fact, $dim(H^2_{tr}(X_t)) = 7$ as we shall see.
\end{remark}
We now describe a particular choice of 7-dimensional basis of 2-cycles that we will use henceforward. First, consider the 1-cycles $\alpha,\beta,\gamma_{-1},\gamma_0,\gamma_1$ over each $E_t$, as described in figure \ref{fig:one_cycles}. Denote by $\delta_1,\delta_2$ the basis for the local system over each point of $E_t$, with $\delta_1\cdotp \delta_2 = 1$. Let's see how they behave under the action of the monodromy, but first we analyze the situation over $\mathbb{P}^1$(with $z^2$-coordinate) before the double cover, i.e on $y^2=x(x-1)(x-z)$, so that we can predict how the cycles change after the double cover. 
\begin{figure}
\centering
\begin{asy}
//settings.render = 8;
//settings.prc = false;

import graph3;
import contour;
size3(8cm);
currentprojection = orthographic(10,1,4);
defaultrender = render(merge = true);
// create torus as surface of rotation
int umax = 40;
int vmax = 40;
surface torus = surface(Circle(c=2Y, r=0.6, normal=X, n=vmax), c=O, axis=Z, n=umax);
torus.ucyclic(true);
torus.vcyclic(true);

pen meshpen = 0.3pt + gray;

draw(torus, surfacepen=material(diffusepen=white+opacity(0.6), emissivepen=white));
for (int u = 0; u < umax; ++u)
  draw(torus.uequals(u), p=meshpen);
for (int v = 0; v < vmax; ++v)
  draw(graph(new triple(real u) {return torus.point(u,v); }, 0, umax, operator ..),
       p=meshpen);

pair inf = (30, 0);
dot(torus.point(inf.x, inf.y), L="$\infty$", align=W);
pair a = (30, 4);
dot(torus.point(a.x, a.y), L="$0$", align=W);
pair b = (30, -4);
dot(torus.point(b.x, b.y), L="$0'$", align=4X);
pair c = (25, 4);
dot(torus.point(c.x, c.y), L="$-1$", align=W);
pair d = (35, -4);
dot(torus.point(d.x, d.y), L="$-1'$");
pair e = (35, 4);
dot(torus.point(e.x, e.y), L="$1$");
pair f = (25, -4);
dot(torus.point(f.x, f.y), L="$1'$", align=W);
path3 abpath(int ucycles, int vcycles, pair one, pair two) {
  pair bshift = (ucycles*umax, vcycles*vmax);
  triple f(real t) {
    pair uv = (1-t)*one + t*(two+bshift);
    return torus.point(uv.x, uv.y);
  }
  return graph(f, 0, 1, operator ..);
}

real linewidth = 0.8pt;

draw(abpath(0,1,a,b), p=linewidth + orange);
draw(abpath(0,0,a,b), p=linewidth + gray + dashed);
draw(abpath(0,1,c,d), p=linewidth + red);
draw(abpath(0,0,c,d), p=linewidth + gray + dashed);
draw(abpath(0,1,e,f), p=linewidth + darkgreen);
draw(abpath(0,0,e,f), p=linewidth + gray + dashed);
\end{asy}
\caption[1-cycles over the Base $E_t$]{1-cycles over the Base $E_t$
\label{fig:one_cycles}}
\end{figure}
The degeneration in this case is a nodal degeneration on $0,1$, the monodromy matrices are then given by the Picard-Lefschetz formula:
\begin{equation}
\begin{split}
    & T_0 = \left(\begin{smallmatrix} 1&2\\ 0&1 \end{smallmatrix}\right)\\
    & T_1 = \left(\begin{smallmatrix} 1&0\\ -2&1 \end{smallmatrix}\right)
\end{split}
\end{equation}
Now consider the $\mathbb{P}^1$ which has $z$ coordinate. In order to make it simply connected, we draw some cuts over it. If we go through paths around $-1,0,1$, as described in figure \ref{fig:z_plane_cycles}, we can look at the image of those cycles under the double cover and see what the monodromy is. For example, as go around -1 on the $z$-plane, the image goes to once around zero, then once around one and one more time around zero again on the $z^2$-plane, hence we can deduce that the local monodromy around $-1$ is $T_0*T_1*T_0^{-1}$. Applying the same reasoning to 0 and 1, we get the resulting monodromies:
\begin{equation}\label{local_monodromies}
\begin{split}
    & \tilde{T_{-1}} = \left(\begin{smallmatrix} -3&8\\ -2&5 \end{smallmatrix}\right)\\
    & \tilde{T_0} = \left(\begin{smallmatrix} 1&4\\ 0&1 \end{smallmatrix}\right)\\
    & \tilde{T_1} = \left(\begin{smallmatrix} 1&0\\ -2&1 \end{smallmatrix}\right)
\end{split}
\end{equation}

The vanishing cycles are: $2\delta_1 + \delta_2$ at -1, $\delta_1$ at 0, $\delta_2$ at 1. Set $\eta_1 = \delta_2$ and $\eta_2 = 2\delta_1 + \delta_2$, so $\eta_1\cdotp \eta_2 = -2$ and the vanishing cycle at 0 is precisely $\delta_1=\frac{1}{2}(\eta_2-\eta_1)$.

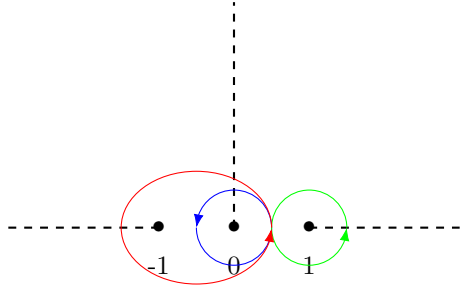
\begin{figure}
\centering
\begin{tikzpicture}
\node at (0,-0.5) {0};
\node at (1,-0.5) {1};
\node at (-1,-0.5) {-1};
\foreach \Point in {(1,0),(-1,0),(0,0)}\node at \Point {\textbullet};
\draw[black,thick,dashed] (1,0) -- (3,0) ;
\draw[black,thick,dashed] (0,0) -- (0,3);
\draw[black,thick,dashed] (-3,0) -- (-1,0);
\draw[blue,arrows={-Latex[length=5]}] (-0.5,0) arc[radius=0.5,start angle=180,end angle=540];
\draw[green,arrows={-Latex[length=5]}] (1.5,0) arc[radius=0.5,start angle=0,end angle=360];
\draw[red,arrows={-Latex[length=5]}] (0.5,0) arc[x radius = 1, y radius = 0.75, start angle= 0, end angle= 359];
\end{tikzpicture}
\caption[Cycles enclosing -1,0 and 1]{Cycles enclosing -1,0 and 1 in $\mathbb{P}^1$ minus the cuts.
\label{fig:z_plane_cycles}}
\end{figure}

Note that these local monodromies are quasi-unipotent, more generally we have:

\begin{theorem}\textbf{(Monodromy Theorem)} \cite{PM}\label{mon_theorem}
Let $p:\Delta^*\to \Gamma\backslash D$ be a variation of weight $w$ Hodge structures over the puncture disk, with monodromy generated by $T$. Then $T$ is quasi-unipotent and has index of unipontency $w+1$.
\end{theorem}
By this theorem, the local monodromies(\ref{local_monodromies}) have to have index of unipontency 2, which is true, as one quickly verifies. Therefore, the log monodromy is just $T-I$, where $T$ is the local monodromy. Recall the formula for the local monodromy in a nodal degeneration:
\begin{proposition}\textbf{(Picard-Lefschetz formula)}\label{plformula}
In a nodal degeneration with vanishing cycles $\delta_i$, the local monodromy is: $$T(x) = x \pm \sum_i \langle x,\delta_i \rangle \delta_i$$
\end{proposition}
\begin{corollary}\label{pfcor}
In a nodal degeneration of weight $1$ Hodge structures, we have up to a sign: $$N(x):= \log T(x) =\sum_i \langle x,\delta_i \rangle \delta_i$$
\end{corollary}
This is equivalent of saying that the image of the log monodromy in a nodal degeneration of weight one Hodge structures is generated by the vanishing cycles.

\section{Construction of the 2-cycles}
We use henceforward the notation $a\times b$ to denote the 2-cycle on $X_t$ obtained by taking the 1-cycle $a$ on a fiber of $\pi_t$ and continuing it along the 1-cycle $b$ on $E_t$. Note that the cycle $a$ has to be monodromy invariant when one goes over $b$, otherwise this definition doesn't make sense.
Now that our notation is established we proceed with the definition of a 7-dimensional subspace of $H^2_{tr}(X_t)$:
\begin{equation}
\begin{aligned}
     & A_1 = \eta_1 \times \alpha  & C_{-1} &= \eta_2 \times \gamma_{-1}\\
     & A_2 = \eta_2 \times \alpha  & C_{0} &= \frac{1}{2}(\eta_2 -\eta_1) \times \gamma_{0}\\
     & B_1 = \eta_1 \times \beta   & C_{1} &= \eta_1 \times \gamma_{1}\\
     & B_2 = \eta_2 \times \beta   & &
\end{aligned}
\end{equation}
Note that, $A_1,A_2,B_1,B_2$ are trivially transcendental, the same is not true for the $C_i$. The reason is that the $C_i$ may--in fact they do--contain algebraic cycles resulting from classes of singular fibers. To overcome this, we have to ``add" enough cycles in order to make all $C_i$ transcendental.

Let's take a closer look at the $C_{-1}$, for example. As we can see from figure \ref{fig:Cm1}, we can pick a cycle equivalent to $C_{-1}$ but with minimal intersection, in other words:
\begin{equation}
\begin{split}
    & C_{-1} \cdotp D_{-} = -1\\
    & C_{-1} \cdotp D_{+} = 1\\
    & C_{-1} \cdotp E_{-} = -1\\
    & C_{-1} \cdotp E_{+} = 1
\end{split}
\end{equation}
\begin{figure}
\centering
\includegraphics[scale=0.7]{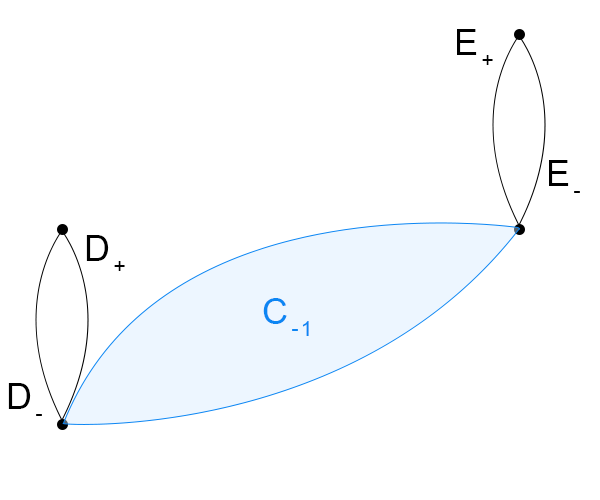}
\caption[The 2-cycle $C_{-1}$]{The 2-cycle $C_{-1}$
\label{fig:Cm1}}
\end{figure}

Now, let's try to eliminate the intersections of $C_{-1}$ with algebraic classes. Start by setting:
\begin{equation}
    \widetilde{C}_{-1} := C_{-1} + aD_{-} + bD_{+} + cE_{-} + dE_{+}
\end{equation}
If $\sigma$ is the class of the zero section, then the transcendental condition reduces to the following system of linear equations:
\begin{equation}
\begin{split}
    &\widetilde{C}_{-1} \cdotp D_{-} = 0\\
    &\widetilde{C}_{-1} \cdotp D_{+} = 0\\
    &\widetilde{C}_{-1} \cdotp \sigma = 0\\
    &\widetilde{C}_{-1} \cdotp E_{-} = 0\\
    &\widetilde{C}_{-1} \cdotp E_{+} = 0
\end{split}
\end{equation}

Without loss of generality we may assume $a=d=0$. Solving the system we get that:
\begin{equation}
    \widetilde{C}_{-1} = C_{-1} + \frac{1}{2}D_{+}  -\frac{1}{2}E_{-}
\end{equation}

By following the exact same reasoning, we deduce that:

\begin{equation}
    \widetilde{C}_{1} = C_{1} + \frac{1}{2}G_{+}  -\frac{1}{2}H_{-}
\end{equation}
where $G_{-}$ and $H_{-}$ are the components of the singular fibers of the endpoints.

Now we address $C_0$, consider the figure \ref{fig:C0}. Following the idea above, we set:
\begin{equation}
    \widetilde{C}_{0} = C_{0} + aL_1 + bL_2 + cL_3 -dF_1 -eF_2 -fF_3
\end{equation}
\begin{figure}
\centering
\includegraphics[scale=0.8]{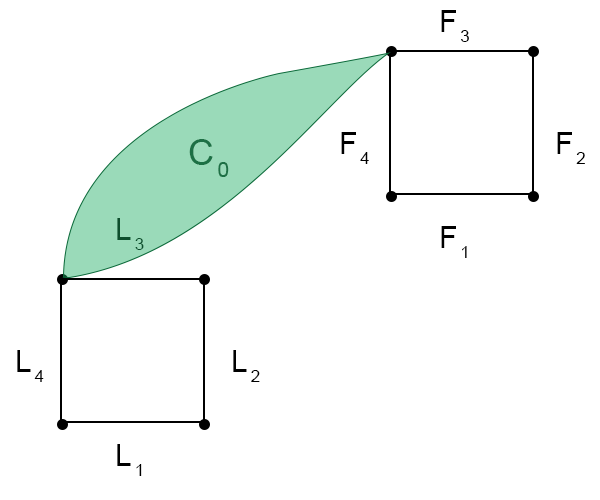}
\caption[The 2-cycle $C_0$]{The 2-cycle $C_0$
\label{fig:C0}}
\end{figure}

We again solve the system of equations required for transcendency:
\begin{equation}
\begin{split}
    &\widetilde{C}_{0} \cdotp L_1 = 0\\
    &\widetilde{C}_{0} \cdotp L_2 = 0\\
    &\widetilde{C}_{0} \cdotp L_3 = 0\\
    &\widetilde{C}_{0} \cdotp F_1 = 0\\
    &\widetilde{C}_{0} \cdotp F_2 = 0\\
    &\widetilde{C}_{0} \cdotp F_3 = 0\\
    &\widetilde{C}_{0} \cdotp \sigma = 0
\end{split}
\end{equation}

The resulting cycle is:
\begin{equation}
    \widetilde{C}_{0} = C_{0} + \frac{3}{4}L_1 + \frac{1}{2}L_2 + \frac{1}{4}L_3 -\frac{3}{4}F_1 -\frac{1}{2}F_2 -\frac{1}{4}F_3
\end{equation}

\section{Computation of the monodromies}
Denote by $V$ the space generated by the transcendental cycles  $(A_1,A_2,$ $B_1,B_2,\widetilde{C_{-1}},\widetilde{C_0},\widetilde{C_1})$. The intersection matrix is:$$Q = \begin{bmatrix}
            0&0&0&2&0&0&0\\
            0&0&-2&0&0&0&0\\
            0&-2&0&0&0&0&0\\
            2&0&0&0&0&0&0\\
            0&0&0&0&-1&1&2\\
            0&0&0&0&1&-1/2&-1\\
            0&0&0&0&2&-1&-1
            \end{bmatrix}$$
Notice that since $det(Q) \neq 0$, we have $dim(V)=7$. Since $V\subset H^2_{tr}(X_t)$, Proposition \ref{trd} implies that $dim H^2_{tr}(X_t) = 7$.

With our transcendental basis $(A_1,A_2,B_1,B_2,\widetilde{C_{-1}},\widetilde{C_0},\widetilde{C_1})$ defined, we now compute the monodromies matrices at the singular points $t=\frac{-2}{3\sqrt{3}},0,\frac{2}{3\sqrt{3}},\infty$. For computational purposes, we will work with figure \ref{fig:cuts} instead of figure \ref{fig:one_cycles}.
\begin{figure}
\centering
\begin{tikzpicture}
\draw[gray,thick,dashed] (2,0) -- (4,0) ;
\draw[gray,thick,dashed] (0,0) -- (0,2);
\draw[gray,thick,dashed] (-4,0) -- (-2,0);
\draw[black,thick,dashed] (0,2) -- (0,3);
\draw[black,thick,dashed] (0,-1.5) -- (0,-2.5);
\draw[black,thick,dashed] (-1,1) -- (1,1);
\draw[blue,thick] (0,-1.85) arc[radius=0.35,start angle=270,end angle=450] arc[radius=0.75,start angle=270,end angle=180] -- (-0.75,1);
\draw[blue,thick,dashed] (-0.75,1) arc[radius=0.35,start angle=0,end angle=180] -- (-1.45,-0.4) arc[radius=1.45,start angle=180,end angle=270];
\draw[Lavender,thick,dashed] (0,-1.75) arc[radius=0.35,start angle=270,end angle=90] arc[radius=0.75,start angle=270,end angle=360] -- (0.75,1);
\draw[Lavender,thick] (0.75,1) arc[radius=0.35,start angle=180,end angle=0] -- (1.45,-0.4) arc[radius=1.35,start angle=360,end angle=270];
\draw[Apricot,thick] (2,0) -- (2,1) arc[radius=0.70,start angle=0,end angle=180];
\draw[Apricot,thick,dashed] (0.6,1) -- (0.6,-0.3) arc[radius=0.6,start angle=360,end angle=180] -- (-0.6,1);
\draw[Apricot,thick] (-0.6,1) arc[radius=0.5,start angle=0,end angle=180] -- (-1.6,-0.7) arc[radius=1.6,start angle=180,end angle=270];
\draw[Apricot,thick,dashed] (0,-2.3) arc[x radius= 2,y radius=2.3,start angle=270,end angle=360];
\draw[OliveGreen,thick,dashed] (-2,0) arc[radius= 2.2,start angle=180,end angle=270];
\draw[OliveGreen,thick] (0,-2.2) arc[x radius = 1.8,y radius=2,start angle=270,end angle=360] -- (1.8,1.8) arc[radius = 0.7,start angle=0,end angle=180] -- (0.4,1);
\draw[OliveGreen,thick,dashed] (0.4,1) -- (0.4,-0.2) arc[radius = 0.4,start angle=360,end angle=180] -- (-0.4,1);
\draw[OliveGreen,thick] (-0.4,1) arc[radius = 0.8,start angle=0,end angle=180] -- (-2,0);
\draw[OliveGreen,thick] (-0.4,1) arc[radius = 0.8,start angle=0,end angle=180] -- (-2,0);
\draw[RubineRed,thick] (0,0) arc[radius = 0.2,start angle=270,end angle=180] -- (-0.2,1);
\draw[RubineRed,thick,dashed] (-0.2,1) arc[radius = 0.8,start angle=0,end angle=180] -- (-1.8,0) arc[x radius = 1.8,y radius = 2,start angle=180,end angle=270];
\draw[RubineRed,thick] (0,-2) arc[radius = 1.6,start angle=270,end angle=360] -- (1.6,1) arc[radius = 0.7,start angle=180,end angle=90] arc[radius = 1.05,start angle=0,end angle=180] -- (0.2,1);
\draw[RubineRed,thick,dashed] (0.2,1) arc[x radius = 0.2,y radius =1,start angle=360,end angle=270];
\node at (0,-0.35) {0};
\node at (2,-0.35) {1};
\node at (-2,-0.35) {-1};
\node[anchor=west] at (0,2) {$\infty$};
\node at (1,1) {$\pmb{\times}$};
\node at (-1,1) {$\pmb{\times}$};
\node at (0,-1.5) {$\pmb{\times}$};
\foreach \Point in {(2,0),(-2,0),(0,0),(0,2)}\node at \Point {\textbullet};
\end{tikzpicture}
\caption[The cuts that define the base curve $E_t$]{The 1-cycles \textcolor{Lavender}{$\alpha$},\textcolor{blue}{$\beta$},\textcolor{OliveGreen}{$\gamma_{-1}$},\textcolor{RubineRed}{$\gamma_0$} and \textcolor{Apricot}{$\gamma_1$} over the Elliptic curve $E_t$
\label{fig:cuts}}
\end{figure}
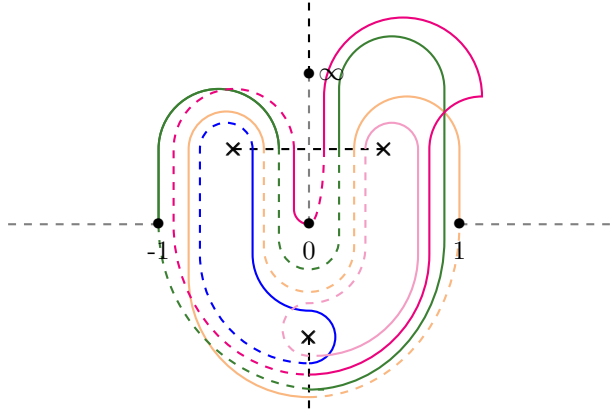
When $t\rightarrow\pm\frac{2}{3\sqrt{3}}$, we have a nodal degeneration on the base curve $E_t$. In figure \ref{fig:cuts}, such degeneration can be described as when the ``x'' of one cut merges itself with an ``x'' of the other cut. 

It's straightforward to conclude that in this case, the $\widetilde{C}_i$ remain unchanged, while in the other cases the cycles over the vanishing cycles remain unchanged. 

Finally, the cycles that do change, do it so according to the Picard-Lefschetz formula (\ref{plformula}), since the degeneration is nodal. In conclusion, we have the following monodromies for $\frac{2}{3\sqrt{3}},\frac{-2}{3\sqrt{3}}$ respectively:
\begin{equation}
    \begin{split}
        M_{+} &= \begin{bmatrix}
            1&0&1&0&0&0&0\\
            0&1&0&1&0&0&0\\
            0&0&1&0&0&0&0\\
            0&0&0&1&0&0&0\\
            0&0&0&0&1&0&0\\
            0&0&0&0&0&1&0\\
            0&0&0&0&0&0&1
            \end{bmatrix}\\
        M_{-} &= \begin{bmatrix}
            1&0&0&0&0&0&0\\
            0&1&0&0&0&0&0\\
            -1&0&1&0&0&0&0\\
            0&-1&0&1&0&0&0\\
            0&0&0&0&1&0&0\\
            0&0&0&0&0&1&0\\
            0&0&0&0&0&0&1
            \end{bmatrix}
    \end{split}
\end{equation}

The situation when $t\rightarrow 0$ is much more subtle. If one looks at figure \ref{fig:cuts}, the endpoints of the cuts behave roughly as $-1-\frac{t}{2},t$ and $1-\frac{t}{2}$, therefore when t go through a path around 0, the endpoints will certain move, but this time not in a nice way as they did in the case above, they will instead make the $\gamma_i$ cycles cross each other and also $\alpha$ and $\beta$. This is the crucial point which results in $G_2$ monodromy, as we shall verify. 

\begin{figure}
\centering
\begin{tikzpicture}
\draw[gray,thick,dashed] (2,0) -- (4,0) ;
\draw[gray,thick,dashed] (0,0) -- (0,2);
\draw[gray,thick,dashed] (-4,0) -- (-2,0);
\draw[black,thick,dashed] (0,2) -- (0,3);
\draw[black,thick,dashed] (0,-1.5) -- (0,-2.5);
\draw[black,thick,dashed] (-1,1) -- (1,1);
\draw[Lavender,thick] (0,-1.9) arc[radius=0.35,start angle=270,end angle=450] arc[radius=0.7,start angle=270,end angle=90] arc[radius=0.7,start angle=90,end angle=0] arc[radius=1,start angle=180,end angle=360] -- (2.7,1) arc[radius=1,start angle=0,end angle=180];
\draw[Lavender,thick,dashed] (0.7,1) arc[radius=0.35,start angle=180,end angle=360] arc[radius=0.4,start angle=180,end angle=0] -- (2.2,-0.35) arc[radius=0.5,start angle=360,end angle=180] arc[x radius = 1.2,y radius=1,start angle=0,end angle=180] arc[x radius = 1.2,y radius=1.55,start angle=180,end angle=270];
\node at (0,-0.35) {0};
\node at (2,-0.35) {1};
\node at (-2,-0.35) {-1};
\node[anchor=west] at (0,2) {$\infty$};
\node at (1,1) {$\pmb{\times}$};
\node at (-1,1) {$\pmb{\times}$};
\node at (0,-1.5) {$\pmb{\times}$};
\foreach \Point in {(2,0),(-2,0),(0,0),(0,2)}\node at \Point {\textbullet};
\end{tikzpicture}
\caption[The deformed $\alpha$ after monodromy]{$\widetilde{\alpha}$, the resulting cycle after monodromy, when $t$ is close to $0$.
\label{fig:talphab}}
\end{figure}
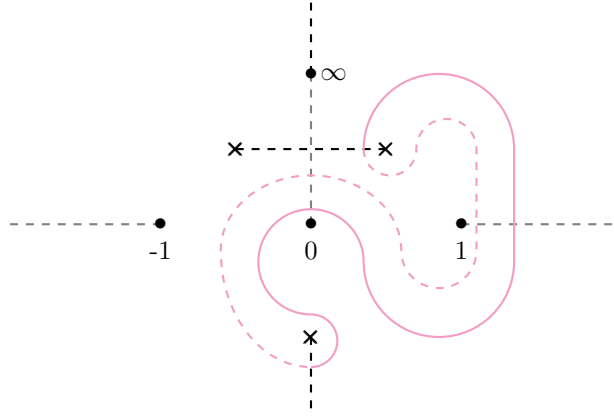
Let's start off by analyzing the resulting cycle $\widetilde{\alpha}$ of the monodromy action on $\alpha$. If we look at figure \ref{fig:talphab}, we see not only $\alpha$ is no longer a vanishing cycle, but also that it crosses the cuts trivializing the local system. What that means basically is that $\eta_1$ and $\eta_2$ might change after monodromy; this is in fact the case, as we shall see.   

Now, consider $\widetilde{\alpha}-\alpha$, as depicted in figure \ref{fig:talpha}. Note that the vanishing cycle at 1 is $\eta_2$, hence any cycle which is the continuation of $\eta_2$ won't have monodromy around 1, so we can simplify $\widetilde{\alpha}-\alpha$ to encircle only 0, and vice-versa. Using the expression for the local monodromies \ref{local_monodromies} and formula \ref{pfcor}, we can compute the resulting 2-cycles for the ones that are over $\alpha$, i.e $A_1,A_2$. Denote by $M_0$ the monodromy at 0, then:
\begin{equation}
    \begin{split}
        & M_0(A_1) = A_1 - 2A_2 + 2B_1 - 2B_2 -4\widetilde{C_0}\\
        & M_0(A_2) = 2A_1 - 3A_2 + 6B_1 - 2B_2 -4\widetilde{C_0} -4\widetilde{C_1}
    \end{split}
\end{equation}

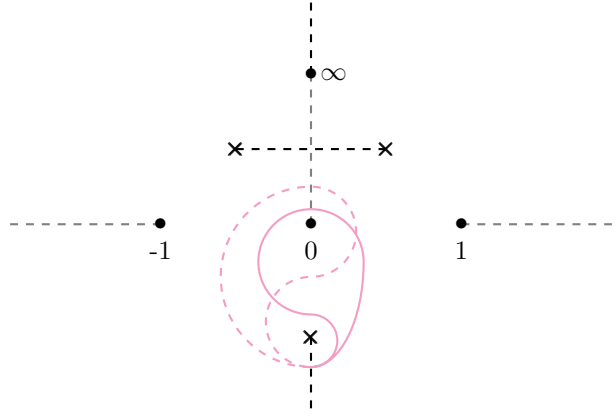
\begin{figure}
\centering
\begin{tikzpicture}
\draw[gray,thick,dashed] (2,0) -- (4,0) ;
\draw[gray,thick,dashed] (0,0) -- (0,2);
\draw[gray,thick,dashed] (-4,0) -- (-2,0);
\draw[black,thick,dashed] (0,2) -- (0,3);
\draw[black,thick,dashed] (0,-1.5) -- (0,-2.5);
\draw[black,thick,dashed] (-1,1) -- (1,1);
\draw[Lavender,thick] (0,-1.9) arc[radius=0.35,start angle=270,end angle=450] arc[radius=0.7,start angle=270,end angle=90] arc[radius=0.7,start angle=90,end angle=0] arc[x radius = 0.7,y radius=1.4,start angle=360,end angle=270];
\draw[Lavender,thick,dashed] (0,-1.9) arc[radius=1.2,start angle=270,end angle=90] arc[radius=0.6,start angle=90,end angle=-90] arc[radius=0.6,start angle=90,end angle=270];
\node at (0,-0.35) {0};
\node at (2,-0.35) {1};
\node at (-2,-0.35) {-1};
\node[anchor=west] at (0,2) {$\infty$};
\node at (1,1) {$\pmb{\times}$};
\node at (-1,1) {$\pmb{\times}$};
\node at (0,-1.5) {$\pmb{\times}$};
\foreach \Point in {(2,0),(-2,0),(0,0),(0,2)}\node at \Point {\textbullet};
\end{tikzpicture}
\caption[$\widetilde{\alpha}-\alpha$]{$\widetilde{\alpha}-\alpha$
\label{fig:talpha}}
\end{figure}

Similarly, we can follow exact the same procedure for $\beta$. We get:
\begin{equation}
    \begin{split}
        & M_0(B_1) = -2A_1 +6A_2 -3B_1 + 2B_2 -4\widetilde{C_{-1}} + 4\widetilde{C_0}\\
        & M_0(B_2) = -2A_1 +2A_2 -2B_1 + B_2 + 4\widetilde{C_0}
    \end{split}
\end{equation}

Now, as figure \ref{fig:tc0} suggest, the case for each $\gamma_i$ is more subtle. Contrary to the $\alpha,\beta$ cases, the 2-cyle $\widetilde{C_0}$, for example, is formed by continuing a 1-cycle that involves both $\eta_1,\eta_2$, therefore we can't ignore any of the points $-1,0,1$ in computing the monodromy. At this point though, we can use the computer and impose the condition that $M_0$ has to preserve $Q$ and solve the system of equations. The result is the following:
\begin{equation}
    \begin{split}
    & M_0(\widetilde{C_0}) = -A_1 +3A_2 -3B_1 + B_2 -2\widetilde{C_{-1}} + \widetilde{C_0} + 2\widetilde{C_1}\\
    & M_0(\widetilde{C_{-1}}) = 2A_1 -4A_2 +6B_1 -2 B_2 +\widetilde{C_{-1}} - 4\widetilde{C_0} - 4\widetilde{C_1}\\
        & M_0(\widetilde{C_1}) = -2A_1 +6A_2 -4B_1 + 2B_2 -4\widetilde{C_{-1}} + 4\widetilde{C_0} + \widetilde{C_1}
    \end{split}
\end{equation}

Now we can write our full monodromy $M_0$:
\begin{equation}
    M_0 = \begin{bmatrix}
            1&2&-2&-2&2&-1&-2\\
            -2&-3&6&2&-4&3&6\\
            2&6&-3&-2&6&-3&-4\\
            -2&-2&2&1&-2&1&2\\
            0&0&-4&0&1&-2&-4\\
            -4&-4&4&4&-4&1&4\\
            0&-4&0&0&-4&2&1
            \end{bmatrix}
\end{equation}

Since we can rearrange the loops around $-1,0,1,\infty$ so that their product is the identity, we naturally get the expression for $M_\infty$ as the inverse of the prodcut $M_{-}\cdotp M_0 \cdotp M_+$, leading to:

\begin{equation}
    M_\infty = \begin{bmatrix}
            0&-4&1&0&-4&2&2\\
            4&0&4&1&-2&2&4\\
            -1&4&-3&-2&6&-3&-4\\
            0&-1&2&1&-2&1&2\\
            -4&0&-4&0&1&-2&-4\\
            0&0&4&4&-4&1&4\\
            0&-4&0&0&-4&2&1
            \end{bmatrix}
\end{equation}
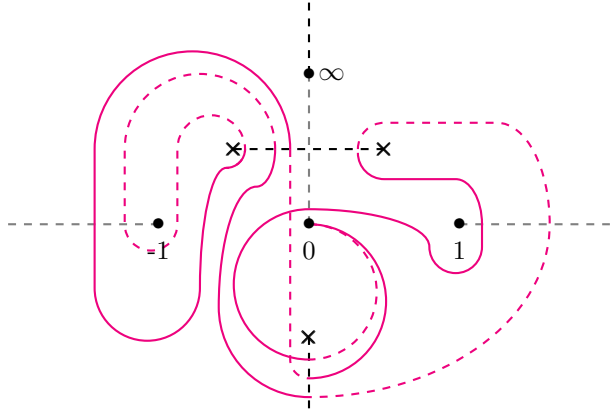
\begin{figure}
\centering
\begin{tikzpicture}
\draw[gray,thick,dashed] (2,0) -- (4,0) ;
\draw[gray,thick,dashed] (0,0) -- (0,2);
\draw[gray,thick,dashed] (-4,0) -- (-2,0);
\draw[black,thick,dashed] (0,2) -- (0,3);
\draw[black,thick,dashed] (0,-1.5) -- (0,-2.5);
\draw[black,thick,dashed] (-1,1) -- (1,1);
\draw[RubineRed,thick,dashed] (0,0) arc[radius = 0.9,start angle=90,end angle=-90];
\draw[RubineRed,thick] (0,-1.8) arc[radius = 1,start angle=270,end angle=90] arc[x radius =1.6,y radius = 0.5,start angle=90,end angle=0] arc[radius = 0.35,start angle=180,end angle=360] -- (2.3,0) arc[x radius =0.35,y radius = 0.6,start angle=0,end angle=90] -- (1,0.6) arc[radius = 0.35,start angle=270,end angle=180];
\draw[RubineRed,thick,dashed] (0.65,1) arc[radius = 0.35,start angle=180,end angle=90] -- (2.5,1.35) arc[x radius =0.7,y radius = 1.35,start angle=90,end angle=0] arc[x radius =3.2,y radius = 2.3,start angle=360,end angle=270];
\draw[RubineRed,thick] (0,-2.3) arc[radius = 1.2,start angle=270,end angle=180] arc[x radius = 0.5,y radius = 1.6,start angle=180,end angle=90] arc[x radius = 0.25,y radius = 0.5,start angle=270,end angle=360];
\draw[RubineRed,thick,dashed] (-0.45,1) arc[radius = 1,start angle=0,end angle=180] -- (-2.45,0) arc[radius = 0.35,start angle=180,end angle=360] -- (-1.75,1) arc[radius = 0.45,start angle=180,end angle=0] arc[radius = 0.25,start angle=360,end angle=270];
\draw[RubineRed,thick] (-0.85,1) arc[radius = 0.25,start angle=360,end angle=270] arc[y radius = 1.6,x radius = 0.35,start angle=90,end angle=180] arc[radius = 0.7,start angle=360,end angle=180] -- (-2.85,1) arc[radius = 1.3,start angle=180,end angle=0];
\draw[RubineRed,thick,dashed] (-0.25,1) -- (-0.25,-1.8) arc[radius = 0.25,start angle=180,end angle=270];
\draw[RubineRed,thick] (0,-2.05) arc[radius = 1.025,start angle=270,end angle=450];
\node at (0,-0.35) {0};
\node at (2,-0.35) {1};
\node at (-2,-0.35) {-1};
\node[anchor=west] at (0,2) {$\infty$};
\node at (1,1) {$\pmb{\times}$};
\node at (-1,1) {$\pmb{\times}$};
\node at (0,-1.5) {$\pmb{\times}$};
\foreach \Point in {(2,0),(-2,0),(0,0),(0,2)}\node at \Point {\textbullet};
\end{tikzpicture}
\caption[The deformed $\gamma_0$ after monodromy]{$\widetilde{\gamma_0}$, the resulting cycle after monodromy around $t=0$.
\label{fig:tc0}}
\end{figure}
\section{The geometric monodromy group}

Recall that by the Monodromy theorem\ref{mon_theorem}, all the monodromies are quasi-unipotent. Hence, all of them have a well-defined logarithm, which we will denote by $N_i:= log(M_i)$, see chapter 1 for a brief review of this topic. 

A quick computation shows that $M_0$ is in fact semi-simple, so the unipotent part $(M_0)_{un}$ is the identity and hence $N_0=0$. The remaining monodromies do have non trivial logarithms: $M_+,M_{-}$ are actually unipotent and $M_{\infty}$ is the only non-unipotent. We can easily check that $M_{\infty}^3$ is unipotent though. 

If $M_\infty = M_s\cdotp M_u$ is the Jordan-Chevalley decomposition and $I$ is the 7x7 identity matrix, then:
\begin{equation}
\begin{split}
    & N_+ = M_+ - I\\
    & N_{-} = M_{-} - I\\
    & N_\infty := log(M_u) = \frac{1}{3}log(M_\infty^3)
\end{split}
\end{equation}

We have the following result concerning the monodromy group of the family $X_t$:
\begin{theorem}\label{g2_main}
The log-monodromies $N_{+},N_{-},N_\infty$ generate $\mathfrak{g}_2$.
\end{theorem}
\begin{proof}
Consider the elements:
\begin{equation}
\begin{aligned}
     & Y_1 = \left[N_{-},N_+\right]  & Y_8 &= \left[Y_5,Y_6\right]\\
     & Y_2 = \left[N_{-},N_\infty\right]  & Y_9 &= \left[N_\infty,Y_5\right]\\
     & Y_3 = \left[N_+,N_\infty\right]  & Y_{10} &= \left[N_\infty,Y_9\right]\\
     & Y_4 = \left[Y_1,Y_2\right] & Y_{11} &= \left[N_\infty,Y_{10}\right]\\
     & Y_5 = \left[Y_1,Y_3\right] & Y_{12} &= \left[N_+,Y_{11}\right]\\
     & Y_6 = \left[Y_2,Y_3\right] & Y_{13} &= \left[N_\infty,Y_{12}\right]\\
     & Y_7 = \left[Y_2,Y_6\right] & Y_{14} &= \left[N_{-},Y_{13}\right]
\end{aligned}
\end{equation}
A quick computation leads us to:
\begin{lemma}
The elements $N_{-},N_+,Y_1,Y_4,Y_5,Y_6,Y_7,Y_8,Y_9,Y_{10},Y_{11},Y_{12},Y_{13},Y_{14}$ are linearly independent over $\mathbb{Q}$.
\end{lemma}

Now define $t_1 := Y_1$ and $t_2 := \left[Y_4,Y_5\right]$, a direct computation gives us that $\left[t_1,t_2 \right]=0$, moreover they both are diagonalizable. Let $ad(.)$ denotes the adjoint representation, if we act through $ad(t_i), i=1,2$, on $\mathfrak{g}$, we get 14 linearly independent (in both cases) eigenvectors with 1-dimensional eigenspaces, moreover we have:
\begin{itemize}
    \item 1 with eigenvalue -2
    \item 4 with eigenvalue -1
    \item 4 with eigenvalue 0
    \item 4 with eigenvalue 1
    \item 1 with eigenvalue 2 
\end{itemize} 
Which are in 1-1 correspondence with the roots of $\mathfrak{g}_2$(see figure \ref{fig:rootsg2}), therefore $\mathfrak{h}:= \langle t_1,t_2\rangle$ is a Cartan subalgebra and $\mathfrak{g} = \mathfrak{g}_2$.
\end{proof}
\begin{figure}
\centering
\includegraphics[scale=0.3]{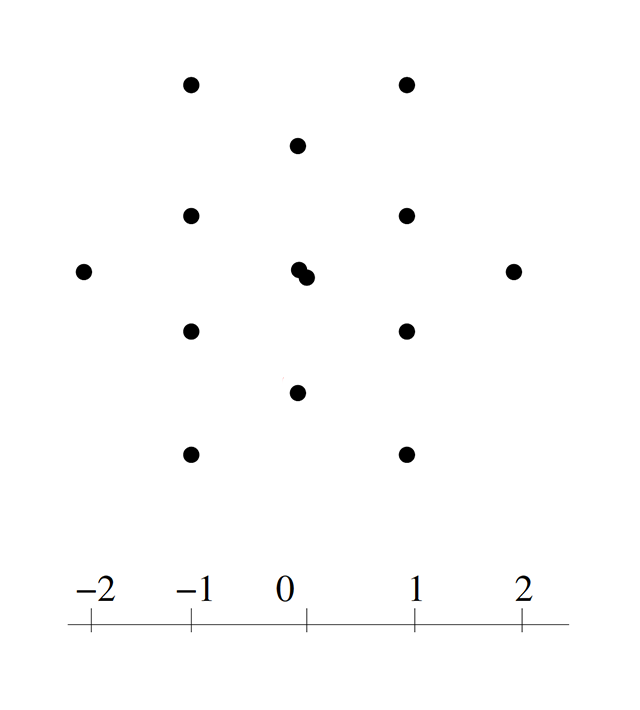}
\caption[Roots of $\mathfrak{g}_2$]{Roots of $\mathfrak{g}_2$\label{fig:rootsg2}}
\end{figure}

This gives us the immediate corollary:
\begin{corollary}
The geometric monodromy group of the family $X_t$ is $G_2$.
\end{corollary}

\begin{remark}
The asymptotics of the period map for this family is given by the limiting mixed Hodge structure(LMHS) at each singularity of the local system.They are easily computed using the expression of the monodromies obtained above, and one hopes that using that fact, one can prove the generic global Torelli theorem for some quotient of this family.
\end{remark}

\section{The fourth family}
Now we analyze the family $X_t$ given by the fourth family in \ref{families}. It is a family of elliptic surfaces obtained by base changing the elliptic surface $$y^2=4x^3+9z^2x^2 + 6zx +1\; (z\in\mathbb{P}^1)$$ by $w^2 = t(z^3-1) + t^2$ $(t\neq 0,1,\infty)$. Repeating the same argument as we did before, now we have 6 fibers of type $I_1$ and one of type $I_{18}$(at $\infty$). Moreover, the local monodromies around $1,r=\frac{-1+\sqrt{3}i}{2}$ and $\bar{r}$ are obtained using the Picard-Lefschetz formula (\ref{plformula}) again, since the degeneration is nodal. We have:

\[ T_{1}=\left( \begin{array}{cc}
1 & 1 \\
0 & 1
\end{array}\right)
, T_r= \left( \begin{array}{cc}
1 & 0 \\
-1 & 1
\end{array} \right)
, T_{\bar{r}}= \left( \begin{array}{cc}
2 & 1 \\
-1 & 0
\end{array} \right)
\]

Let $\delta,\eta$ be the basis for the local system with $\delta\cdot\eta=1$, then the vanishing cycles are $\delta$ at $1$, $\eta$ at $r$ and $\delta - \eta$ at $\bar{r}$. Let $\gamma_i$ the path connecting the 2 points in $E_t$ which are the pre-image of `$i$', and $\alpha,\beta$ the basis for the first homology of the Elliptic curve $E_t$, see Figure \ref{fig:et2}. By using the same idea of the first case, we define the 7 cycles:

\begin{equation}
\begin{aligned}
     & A_1 = \delta \times \alpha  & C_{1} &= \delta \times \gamma_{1}\\
     & A_2 = \eta \times \alpha  & C_{r} &= \eta \times \gamma_{r}\\
     & B_1 = \delta \times \beta   & C_{\bar{r}} &= (\delta - \eta )\times \gamma_{\bar{r}}\\
     & B_2 = \eta \times \beta   & &
\end{aligned}
\end{equation}

Again, the cycles $A_1,A_2,B_1,B_2$ are trivially transcendental but the $C$s are not. Let $E_i,E_i'$ be the two nodal curves over the end of $\gamma_i$, then the cycles $\widetilde{C_i} = C_i + E_i+E_i'-\sigma$, for $i=1,r,\bar{r}$ and $\sigma$ the zero section. The intersection matrix with respect to the basis $(A_1,A_2,B_1,B_2,\widetilde{C_1},\widetilde{C_r},\widetilde{C_{\bar{r}}})$ is:
$$Q = \begin{bmatrix}
            0&0&0&-1&0&0&0\\
            0&0&1&0&0&0&0\\
            0&1&0&0&0&0&0\\
            -1&0&0&0&0&0&0\\
            0&0&0&0&-2&1&-1\\
            0&0&0&0&1&-2&1\\
            0&0&0&0&-1&1&-2
            \end{bmatrix}$$
In particular, $\det Q \neq 0$ again and the cycles above are indeed a basis. The monodromy $M_1$ around $t=1$ is really simple, the crosses $\pmb{\times}$ rotate counter-clockwise around each other until they meet the spot the other was. As the picture \ref{fig:et2} suggests, the $\widetilde{C_i}$ don't move, so we can ignore them. A quick computation using \cite{sage} leads us to:
$$\alpha \rightarrow \alpha-\beta,\beta \rightarrow \alpha$$
As $t\to 1$, the branch cuts of $E_t$ don't cross the cuts trivializing the $\mathbb{P}^1$ local system, hence the cycles $\{\delta,\eta\}$ don't change and our matrix becomes:
$$M_1 = \begin{bmatrix}
            1&0&1&0&0&0&0\\
            0&1&0&1&0&0&0\\
            -1&0&0&0&0&0&0\\
            0&-1&0&0&0&0&0\\
            0&0&0&0&1&0&0\\
            0&0&0&0&0&1&0\\
            0&0&0&0&0&0&1
            \end{bmatrix}\\$$
The Jordan Block of this matrix is:
$$\begin{bmatrix}
            1&0&0&0&0&0&0\\
            0&1&0&0&0&0&0\\
            0&0&1&0&0&0&0\\
            0&0&0&\zeta_6&0&0&0\\
            0&0&0&0&\zeta_6&0&0\\
            0&0&0&0&0&\bar{\zeta_6}&0\\
            0&0&0&0&0&0&\bar{\zeta_6}
            \end{bmatrix}\\$$
where $\zeta_6$ is the 6-th root of unity. Our result matches Katz's description \cite{KATZ} of the Jordan  block structure of the monodromy around 1.
\begin{figure}
\centering
\begin{tikzpicture}
\draw[gray,thick,dashed] (2,0) -- (4,0) ;
\draw[gray,thick,dashed] (-4,1) -- (-2,1);
\draw[gray,thick,dashed] (-4,-1) -- (-2,-1);
\draw[black,thick] (-1,1) -- (-1,-1);
\draw[black,thick] (1,0) -- (1,2);
\draw[Goldenrod,thick] (-.65,1) arc[radius = .35,start angle=0,end angle=180] -- (-1.35,-1) arc[radius = .35,start angle=180,end angle=360] -- (-.65,1);
\draw[Salmon,thick] (1,.35) arc[radius = .35,start angle=90,end angle=-90] -- (-1,.65);
\draw[Salmon,thick,dashed] (-1,.65) arc[radius = .35,start angle=270,end angle=90] -- (1,.35);
\draw[Mulberry,thick] (-2,1) -- (-1.7,-1) arc[y radius = 0.8,x radius = 1.6,start angle=180,end angle=360] -- (1.5,1) arc[radius = .5,start angle=0,end angle=90];
\draw[Mulberry,thick,dashed] (1,1.5) -- (-1.5,1.7) -- (-2,1);
\draw[ProcessBlue,thick] (-2,-1) arc[y radius = 1.2,x radius = 1.9,start angle=180,end angle=360] -- (1.5,1) arc[x radius = .5,y radius = .35,start angle=0,end angle=90];
\draw[ProcessBlue,thick,dashed] (1,1.35) -- (-1.5,1.55) -- (-2,-1);
\draw[PineGreen,thick] (2,0) -- (1.5,1) arc[y radius = .25,x radius = .5,start angle=0,end angle=90];
\draw[PineGreen,thick,dashed] (1,1.25) -- (-1.4,1.45) -- (-1.5,-1) arc[radius =1.8,start angle=180,end angle=360] -- (2,0);
\node at (2,-0.35) {1};
\node at (-2,.65) {$r$};
\node at (-2,-1.35) {$\bar{r}$};
\node at (-1,1) {$\pmb{\times}$};
\node at (-1,-1) {$\pmb{\times}$};
\node at (1,0) {$\pmb{\times}$};
\node at (1.35,2.35) {$\infty$};
\foreach \Point in {(2,0),(-2,1),(-2,-1),(1,2)}\node at \Point {\textbullet};
\end{tikzpicture}
\caption[1-cycles over the Base $E_t$]{The 1-cycles \textcolor{Goldenrod}{$\alpha$},\textcolor{Salmon}{$\beta$},\textcolor{Mulberry}{$\gamma_{r}$},\textcolor{ProcessBlue}{$\gamma_{\bar{r}}$} and \textcolor{PineGreen}{$\gamma_1$} over the Elliptic curve $E_t$
\label{fig:et2}}
\end{figure}
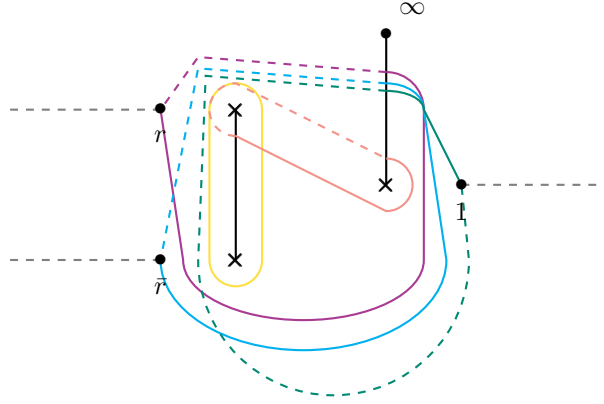

On the other hand, when $t\to 0$ the situation is really complicated. As before, the branch cuts cross the cuts of the
local system, the $\pmb{\times}$ in Figure \ref{fig:et2} encircle $r,\bar{r},1$ counter-clockwise once. 

According to Katz \cite{KATZ}, the Jordan Block of $M_0$ should be
$$\begin{bmatrix}
            -1&0&0&0&0&0&0\\
            0&-1&0&0&0&0&0\\
            0&0&-1&0&0&0&0\\
            0&0&0&-1&0&0&0\\
            0&0&0&0&1&0&0\\
            0&0&0&0&0&1&0\\
            0&0&0&0&0&0&1
            \end{bmatrix}\\$$

Note, in particular, that $M_0$ has to be of finite order.

We couldn't find this monodromy, but by looking at the representation description given by Katz, we conjecture that the monodromy group is reducible and is either $SL(2)\times SL(2)\times SL(2)$ or $SL(2)\times SL(2)\times U(1)$.
\bibliographystyle{amsplain}

\end{document}